\documentclass[reqno]{amsart}

\usepackage[mathscr]{eucal}
\usepackage{amssymb,verbatim}

\newtheorem{thm}{Theorem}[section]
\newtheorem{lem}[thm]{Lemma}
\newtheorem{cor}[thm]{Corollary}
\newtheorem{pro}[thm]{Proposition}
\newtheorem{eg}[thm]{Example}
\newtheorem{defn}[thm]{Definition}
\newtheorem{notation}[thm]{Notation}
\newtheorem{remark}[thm]{Remark}
\newtheorem{problem}[thm]{Problem}

\newcommand{\e}{\mathsf e}
\newcommand{\f}{\mathsf f}
\newcommand{\h}{\mathsf h}
\newcommand{\g}{\mathsf g}

\newcommand{\Wsub}[1]{\mathbf W\kern -2pt_{#1}}
\newcommand{\Wsubit}[1]{W\kern -2pt_{#1}}

\newcommand{\up}[1]{\textup{#1}}

\newcommand{\mc}{\mathcal}

\newcommand{\bea}{\begin{eqnarray*}}
\newcommand{\eea}{\end{eqnarray*}}

\newcommand{\ben}{\begin{enumerate}}
\newcommand{\een}{\end{enumerate}}

\newcommand{\bi}{\begin{itemize}}
\newcommand{\ei}{\end{itemize}}

\newcommand{\Pa}{{\mathcal P}}
\newcommand{\id}{\mathsf{id}}
\newcommand{\dom}{\operatorname{dom}}

\begin{document}
\title[Monoids with tests]{Monoids with tests and the algebra of possibly non-halting programs.}
\author{Marcel Jackson and Tim Stokes}

\address{Department of Mathematics and Statistics, La Trobe University, Victoria, Australia}
\email{m.g.jackson@latrobe.edu.au}
\address{Department of Mathematics, University of Waikato, New Zealand}
\email{stokes@math.waikato.ac.nz}



 \thanks{The first author was supported by ARC Future Fellowship FT120100666 and ARC Discovery Project DP1094578.}
\begin{abstract}
We study the algebraic theory of computable  functions, which can be viewed as arising from possibly
non-halting computer programs or algorithms, acting on some state space, equipped with operations of composition, {\em if-then-else} and {\em while-do} defined in terms of a Boolean algebra of conditions.
It has previously been shown that there is no finite axiomatisation of
algebras of partial functions under these operations alone, and this
holds even if one restricts attention to transformations (representing halting programs) rather than partial functions, and omits {\em while-do} from the signature.  In the halting case, there is a
natural ``fix", which is to allow composition of halting programs with conditions, and then the resulting algebras admit a finite axiomatisation.  In the current setting such compositions are not possible, but by extending the notion of {\em if-then-else}, we are able to give finite axiomatisations of the resulting algebras of (partial) functions, with {\em while-do} in the signature if the state space is assumed finite.  The axiomatisations are extended to consider the partial predicate of equality.  All algebras considered turn out to be enrichments of the notion of a (one-sided) restriction semigroup.

\end{abstract}

\maketitle

\section{Motivation and definitions}

\subsection{Some terminology}
Let $X,Y$ be sets.  A {\em function $X\rightarrow Y$} is a partial map from a subset of $X$ into $Y$, and the set of all such is
denoted $\Pa(X,Y)$.  If $Y=X$ this is denoted $\Pa(X)$, a semigroup under composition (read left to right, so that $(fg)(x)=g(f(x))$ for all $f,g\in \Pa(X)$ and $x\in X$), and an element of $\Pa(X)$ is called a {\em function on $X$}.  Because $X$ is usually some fixed ``global domain'', we use the name \emph{domain} of $f$ (written $\dom(f)$) to denote the subset of points at which $f$ is actually defined.  The {\em identity map} $1_X$ on $X$ is the total function on $X$ that fixes every $x\in X$, and the {\em null map} $0_X$ is the function on $X$ with empty domain; they are respectively identity and zero elements in
the semigroup $\Pa(X)$.  The subscript $X$ will be omitted where the
choice is clear.  A {\em transformation on $X$} is an everywhere-defined function in $\Pa(X)$ (that is, having domain all of $X$); the set of all such is ${\mathcal T}(X)$, a submonoid of $\Pa(X)$.  Transformations are also known as total functions.

A {\em predicate} on $X$ is an everywhere-defined function $X\rightarrow \{T,F\}$; the set of all such is $2^X$, a Boolean algebra under the
usual logical connectives.  Denote by $I(X)$ the monoid of
restrictions of the identity function under composition; it is isomorphic to the semilattice $(2^X,\cap)$.

\subsection{Computable functions}
Many authors have investigated algebraic foundations for facets of the theory of computer programs: amongst others we have in mind are 
\emph{sum-ordered (partial) semirings} (Manes and Benson \cite{manesben}), 
\emph{dynamic algebras} (Pratt and others, see \cite{pra}), \emph{Kleene algebras with tests} (KAT) (Kozen \cite{kozenhoare}),  \emph{Kleene algebra with domain} (KAD; Desharnais, M\"oller and Struth~\cite{KAD}; see also Hirsch and Mikul\'as \cite{hirmik} and Desharnais, Jipsen and Struth~\cite{DJS}),  \emph{modal semirings} (M\"oller and Struth \cite{M&S}), 
\emph{refinement algebras} (von Wright \cite{vonW}),
\emph{correctness algebras} (Guttmann \cite{guttmann}),
as well as the authors' own contributions such as \emph{modal restriction semigroups} \cite{modrest}.   Approaches based on the full Tarski algebra of relations are detailed in Maddux \cite{mad}. 

These algebraic approaches have substantial connections with classical program logics, such as Hoare logic and various propositional dynamic logics, enabling straightforward algebraic equational reasoning to supplant the conventional logical approach.  Indeed, with the exception of KAT (which does not allow for domain information), the algebraic systems mentioned are designed to interpret at least the modal logic part of dynamic logic.   However the family of computable (that is, partial recursive) functions on a set $X$ is not typically a model of any of these algebraic systems.

First, in all but the case of modal restriction semigroups, these algebraic systems allow for union and usually reflexive transitive closure, reflecting the expressiveness of the associated logical systems.  This effectively forces a relational semantics rather than a functional one.  However, if one sticks to the basic {\em if-then-else} and {\em while-do} constructs, there is no need to use a relational semantics: partial functions suffice.

Additionally, many of these systems admit a notion of domain complementation, which is not compatible with the computability assumption.  More specifically, in dynamic logic, the proposition $[f]\text{false}$ (``necessarily false'' as determined by $f$ as a modal relation) holds exactly on the points at which the program $f$ fails to halt, and in general this is clearly not computable.  Moreover, in the algebraic formulations mentioned (as well as in dynamic logic itself), these propositions may themselves then become test conditions within other programs.  Thus constructions such as ``while program $f$ does not halt do program $g$'' can be expressed, and indeed give rise to their own nonhalting proposition and so on.

Test conditions arising in actual programs are typically Boolean combinations of basic tests, and certainly not statements on halting conditions.
The goal of the present article is to present algebraic systems that are rich enough to express standard deterministic programming connectives, such as {\em if-then-else}, as well as tests for equality and non-equality of variable values, yet does not permit the leaching of statements on halting into the test type.
The set of all partial recursive functions on $\mathbb{N}$ will provide a model,  
in contrast to the more common approach in which binary relations model programs.  
(There are of course other approaches to program algebra, such as predicate transformer semantics, but these work at a lower level in which assignments are modelled for example.)

The main results consist of finite axiomatisations for these systems, which we are able to show are complete (for the full first order theory, not just the equational fragment) with respect to suitable functional semantics.  Operations modelling looping are also considered, but for these we are unable to obtain finite axiomatizations, instead making do with axioms at least guaranteeing that certain desirable properties are satisfied,
including completeness for finite (and even periodic) algebras.

The basic approach shares features of both the KAT approach of \cite{kozenhoare} and the modal restriction semigroups with preferential join approach of \cite{modrest}.  Our algebraic systems consist of a semigroup of functions with an embedded sort $B$ of ``tests'', which will form a Boolean algebra in which the meet operation is just the underlying multiplication of the semigroup (so that the tests will form a subsemilattice: an idempotent and commutative subsemigroup), and with a Boolean complementation operation (which is not defined outside of the test sort).  

So far this is identical to the union- and star-free fragment of the algebraic systems considered in KAT.  However, we also consider unary operations modelling domain (as in KAD, or modal restriction semigroups and its predecessors \cite{csg}) as well as various equality test, {\em if-then-else} and {\em while-do} constructions.  All elements of the test sort $B$ are fixed by domain, but in general domain elements form a strictly larger subsemilattice than $B$.  
This reflects the fact that for us, tests are to be viewed as conditions in programs (and indeed as special types of programs themselves) rather than as assertions: we have no need to generate weakest preconditions for example, hence no need to view general domain elements as tests, and indeed no need for domain complement (antidomain) at all.

Another significant difference with KAT and KAD is that we use partial functions as our semantics of programs rather than binary relations.  This relates to the fact that we seek to model programs themselves rather than assertions about programs: we have no need of union or Kleene closure, which feature in dynamic logic for example, so we do not need binary relations either.  (Binary relations are of course closed under both operations, while partial functions are closed under neither.)

The computable functions on a set $X$ will form a model of each of the systems we consider, with $B$ modelling a Boolean algebra of identity maps with \emph{recursive} domains, and with general domain elements corresponding to the identity map on recursively enumerable subsets of $X$ (which are exactly the domains of computable functions).  

For the remainder of this section we further motivate and then carefully define the additional operations used to model {\em if-then-else} and other constructions.

\subsection{Extending {\em if-then-else}}

Operations modelling the {\em if-then-else} command of imperative programming languages have been modelled algebraically by a number of authors.
The idea is to model programs as functions $X\rightarrow Y$, (or even binary relations in $X\times Y$, though we do not consider these here), and then to define, for any $f,g\in \Pa(X,Y)$, and any predicate $\alpha\in 2^X$, 
\begin{align}
&\bullet \quad (\mbox{if }\alpha\mbox{ then }f\mbox{ else }g)(x)=\alpha[f,g](x):=\begin{cases} f(x)&\mbox{ if }x\mbox{ satisfies } \alpha,\\
g(x)&\mbox{ otherwise,} \label{ite}
\end{cases}
\end{align}
for all $x\in X$.  Then $\alpha[f,g]\in {\mathcal P}(X,Y)$ equals $f(x)$ when $\alpha(x)$ is true, and is $g(x)$ otherwise.  Of course this operation is motivated by the {\em if-then-else} connective of computer programs.

The case in which $Y=X$ is the one of chief interest to us here, although the more general setting saw much early work; see McCarthy \cite{mc}, Bergman \cite{berg}, Manes \cite{manes}, as well as the related work of Bloom and Tindell \cite{blotin}, Meklar and Nelson \cite{meknel}, and Guessarian and Meseguer \cite{guemes}.

Approaches based on a relational semantics for programs, but in which $Y=X$ so that composition is available as an operation, often make use of {\em test semirings}; see \cite{kozenhoare}, which may be viewed as ``Kleene algebras with tests" that lack the Kleene closure operation.  They model composition and union of programs (themselves modelled as binary relations on some space), each equipped with a sub-Boolean algebra of ``tests" (which are viewed as programs induced by Boolean tests that fix every element of their domains).  Then for programs $f,g$ and a test $\alpha$, it is possible to write
\begin{align}
&\bullet \quad \alpha[f,g]=\alpha f\cup \alpha'g,&
\end{align}
where $\alpha'$ is the ``complement" of the test $\alpha$.  

In \cite{ITEIJAC}, a functional semantics for halting programs was considered, where structures $(S,B)$ of the following form featured: $S$ is a semigroup of functions on some set $X$ (a subsemigroup of ${\mathcal P}(X)$), $B$ is some Boolean algebra of predicates on $X$ (a subalgebra of $2^X$), and $S$ is also closed under the
{\em if-then-else} operations associated with the elements of $B$.  Based on an idea developed for \cite{jacsto:amon}, it was shown that the class of such two-sorted algebras was not finitely axiomatizable, but that adding the following two-sorted operation to the signature gives finitely axiomatized structures in the case of transformations (total functions):
$\cdot: S\times B\rightarrow B$ such that $s\cdot\alpha$ is the functional composite
of $s,\alpha$. This mixed operation is well-motivated since it gives rise to a ``computable condition": is $\alpha$ true after $s$ is executed?   This composite $s\cdot \alpha$ is analogous to the ``Piercean operator'' of Boolean modules, in the sense of Brink~\cite{bri}.  A complete axiomatization in terms of finitely many equations is given in \cite{ITEIJAC}.  Also axiomatized is the predicate of equality of transformations.

The signature used in \cite{ITEIJAC} to yield a finite axiomatization for transformations cannot be used for partial functions, because the function-predicate composition operation $\cdot: S\times B\rightarrow B$ is not defined: $s\cdot\alpha$ will not be a predicate.  However, it will be a partial or ``possibly non-halting" predicate.  As mentioned, such ``non-halting tests" are considered by Manes in \cite{manes}.  This suggests the possibility of generalising the algebra of tests to admit the possibility that they do not halt, as in \cite{manes}, where the Boolean algebra is replaced by a ``C-algebra" in which a third alternative ``does not halt" is added to ``true" and ``false".  

However, we show here that in a setting in which composition is present, it is not necessary to introduce non-halting tests as separate entities: instead we retain only Boolean conditions, modelling elementary tests, and we generalise the {\em if-then-else} operations themselves, by defining the mixed quaternary operation $S\times B\times S\times S\rightarrow S$ of {\em extended if-then-else}, given by:
\begin{align}
&\bullet \quad (f,\alpha)[g,h](x):=\begin{cases} g(x)&\mbox{ if }f(x)\mbox{ satisfies } \alpha,\\
h(x)&\mbox{ if $f(x)$ satisfies $\alpha'$,}
\end{cases}&
\end{align}
for all $f,g,h\in S$ and $\alpha\in B$.  Note that if $x$ is outside of the domain of $f$ (that is, $f$ does not halt when executed at $x$), then the test $(f,\alpha)[g,h]$ is also undefined.  For transformations (where $f$ is defined everywhere), this yields $(f\cdot\alpha)[g,h]$ as in~\cite{ITEIJAC}.  In the present article we will obtain a finite axiomatization of the resulting algebras, which properly generalises the axiomatizations provided in \cite{ITEIJAC}.  Our approach is based on first axiomatizing restriction semigroups equipped with a sub-Boolean algebra of ``tests elements".

Also considered in \cite{ITEIJAC} is a predicate-valued operation of transformation equality, $*:S\times S\rightarrow B$, with $(f*g)(x)$ {\em true} if and only if $f(x)=g(x)$ and {\em false} otherwise.  Again, for partial functions the operation will not give a
predicate in general since it is undefined at $x\in X$ if either function is undefined at $x$.  In the present article, we are still able to axiomatize the quaternary operation $S^4\rightarrow S$, which we call {\em weak comparison}, given by
\begin{align}
&\bullet \quad(f=g)[h,k](x):=\begin{cases} h(x)&\mbox{ if }f(x)=g(x),\\
k(x)&\mbox{ if }f(x)\neq g(x).
\end{cases}&
\end{align}
Note that for $f(x)\neq g(x)$ to be true, we require that both $f$ and $g$ are defined at $x$, but return different values: if one or both of $f$ and $g$ are undefined at $x$, then so is $(f=g)[h,k]$.  For transformations (where $f$ and $g$ are always defined), this coincides with $(f*g)[h,k]$ considered in~\cite{ITEIJAC}. 
We axiomatize weak comparison in monoids of functions with zero, both in the presence of the extended {\em if-then-else} operations just defined but also on its own, along the way axiomatizing some less rich (hence more general) structures.

From this point, it is possible to define ``non-halting conditions" in terms of the new quaternary operations, and then to define logical connectives on them, giving a C-algebra as in \cite{manes} into which the given Boolean algebra $B$ embeds.  These induced connectives are natural in the setting of actual programming languages, as is discussed at length in \cite{manes}.

\subsection{Looping, {\em skip}, {\em abort} and tests}

To model looping using halting conditions, given $f\in \Pa(X)$ and $\alpha\in 2^X$, we define $(\alpha:f)=(${\em while }$\alpha$ {\em do }$f)$ to be, for any $x\in X$,
\begin{align}
&\bullet \quad (\alpha:f)(x)
:=\begin{cases} f^n(x)&\parbox{7cm}{ if $f^m(x)$ satisfies $\alpha$ for all $0\leq m< n$ but $f^{n}(x)$ does not,}\\
\mbox{undefined}&\mbox{ if }f^n(x)\mbox{ satisfies } \alpha\mbox{ for all }n\geq 0,
\end{cases}&
\end{align}
where we deine $f^0(x)=x$ for all $x\in X$.
So $(\alpha:f)$ acts by repeatedly iterating $f$ until the result no longer satisfies $\alpha$; if it always does, the loop
does not halt and there is no output.

We do not consider {\em while-do} in detail here, although  it is the natural source of non-halting.  Moreover, the functions $1$ (the identity) and $0$ (the empty function) are easily motivated in terms of the programs {\em skip} and {\em abort} respectively, which for example arise as
({\em while}  {\bf false} {\em do} $P$) for any program $P$, and ({\em while}  {\bf true} {\em do} {\em skip}), respectively. 

The presence of {\em if-then-else}, {\em skip} and {\em abort} now forces a copy of the conditions in $B$ into the program type $P$, via the correspondences:
$$\alpha\leftrightarrow \alpha[1,0],\ \
\alpha'\leftrightarrow \alpha[0,1].$$
In this way, each condition $\alpha\in B$ is manifest as a function which is a restriction of the identity $1_X$ to the truth
set of $\alpha$, which we call a {\em test}.  It is easy to see that conjunction of conditions arises as composition of their
corresponding tests, and it only remains to view complementation as an operation on restrictions of the identity, an operation
we call {\em test complement}.  

This embedding
is not possible in the case of halting programs considered in~\cite{ITEIJAC}, where only everywhere-defined transformations
are considered, although it is standard in the test semiring approaches of Kozen and others.  These approaches involve use of the ``Kleene closure" or ``asterate" operation (modelling reflexive transitive closure of binary relations) to model iteration. Objects called `Kleene algebras with tests" are used---these are Kleene algebras which are simultaneously test semirings; see \cite{kozenhoare} for example. In these, one may define 
$$(\alpha:f)=(\alpha f)^*\alpha'.$$   
When applied to the Kleene algebra with tests of binary relations on a set, this formula agrees with ours if $f$ is a function.

\subsection{Extended {\em while-do}}

We may extend {\em while-do} in the same way we have extended {\em if-then-else}.  By direct analogy, we define {\em extended while-do} on ${\mathcal P}(X)$ as follows:
given $f,g\in \Pa(X)$ and $\alpha\in 2^X$, we define, for any $x\in X$,
\begin{align}
&\bullet \quad((f,\alpha):g)(x):=\begin{cases} g^n(x)&\parbox{6cm}{if $f(g^m(x))$ satisfies $\alpha$ for all $0\leq m<n$ but $f(g^{n}(x))$ does not,}\\
\mbox{undefined}&\mbox{ if }f(g^n(x))\mbox{ satisfies } \alpha\mbox{ for all }n\geq 0.
\end{cases}
\end{align}
This says ``keep applying $g$ as long as the result satisfies $\alpha$ when $f$ is applied to it".  So if $f=1$, we recover $(\alpha:g)$.  
Note that $(1,\alpha)[b,c]$ equals the usual {\em if-then-else} operation $\alpha[b,c]$ as defined previously.
It would be possible to define a form of extended {\em while-do} that makes reference to the equality predicate, but we do not pursue this here.

\subsection{Domain}

In $\Pa(X)$, the derived unary operation $D:S\rightarrow S$ is a very natural one.  For partial functions in ${\mathcal P}(X)$, its formal definition is as follows:
\begin{align}
&\bullet \quad D(f):=\{(x,x)\mid (x,y)\in f\mbox{ for some }y\in X\},&
\end{align}
the restriction of the identity function to the domain of $f$.
In extended {\em if-then-else} algebras, it is given by $D(f)=(f,1)[1,1]$.  Conversely, for functions we have that
\begin{align}
&\bullet \quad (f,\alpha)[g,h]=D(f\alpha)g\cup D(f\alpha')h.&
\end{align}
Note also that $D(f)=(f=f)[1,0]$, so domain is again a derived operation in the presence of weak comparison.  Domain is also natural in the setting of the dynamic algebras of \cite{pra:ind}, and the domain semirings and Kleene algebras with domain as in~\cite{KAD}, and the modal semirings of \cite{M&S}.  In most of this work, which has a relational semantics for programs, domain complement is an allowable construct (although this is not the case in \cite{DJS} for example, where one has a distributive lattice of tests). Our approach will be to first axiomatize domain (but not domain complement) in the presence of tests, then to add axioms for extended {\em if-then-else} and weak comparison.

The domain operation $D$ has been considered by many authors in the setting of both partial transformations and binary relations.  In the functional case, the associated class of unary semigroups has a finite axiomatization (given below), and is now often called the class of {\em \up(left\up) restriction semigroups}.  The $D$ operation was considered by Trokhimenko \cite{trok}, who axiomatized a multiplace version in the functional case; the same characterization (at least for single place functions) can be obtained by adapting the earlier work of Schweizer and Sklar~\cite{SS}, and has been rediscovered in various guises by subsequent
authors, including the present authors~\cite{csg} (where they arise as \emph{twisted left closure semigroups}) and
Manes~\cite{manes} (where they arise as \emph{guarded semigroups}).  Restriction semigroups are closely related to weakly right ample
semigroups (see~\cite{gomgou99} and Fountain~\cite{fou} for example).
See also~\cite{coclac}, where a category theoretic version is considered, and~\cite{gouhol2}, where Gould and Hollings
present a variant of the ESN theorem of inverse semigroup theory \cite{lawson} applying to left restriction semigroups.  (These last two sources use the ``restriction" epithet we use here, which is becoming standard in the literature.)

Restriction semigroups in which there is a notion of complementation of domain elements are considered in \cite{modrest}.  As already discussed, domain complement is not in general a computable function and so is not in the signature of the algebras considered here.  However, the main results in \cite{modrest} may be considered consequences of results in the current article, applying to the case in which the set of test elements coincides with the image of $D$ (every domain element is a test), a point to which we return below.

The operations of composition and domain for binary relations on a set are considered by M\"{o}ller and Struth in \cite{M&S},
in the setting of {\em modal semirings}, where an operation modelling relational union is also present.  Such algebras are test semirings and so permit expression of {\em if-then-else}.  Kleene algebras with tests also possess a Kleene (reflexive transitive) closure unary operation (for modelling program iteration).  These systems have finite axiomatisations that are complete with respect to equational properties (holding in the relational semantics), but not with respect to wider properties.  With reflexive transitive closure, an incompleteness theorem is known: the quasi-equational theory of relational models is $\Pi_1^1$-complete (Hardin and Kozen~\cite{harkoz}) and so no recursive axiomatisation can exist.   Even without the reflexive transitive closure operation, these systems are  very unlikely to have finite complete axiomatisations (with respect to relational semantics);~see \cite{and1,and2,andmik}.

\section{Axiomatizations}

\subsection{Axioms for extended {\em if-then-else}}

As discussed, we might as well assume from the beginning that the unary operation $D$ is present in our signature, since it is expressible in terms of both extended {\em if-then-else} and weak comparison, and has in any case formed the basis of various algebraic approaches based on a relational semantics.  In fact, extended {\em if-then-else} may be completely specified in the presence of $D$ by the following laws:
\begin{align}
&\bullet \quad D(s\alpha)((s,\alpha)[t,u])=D(s\alpha)t &\label{eq:EITE2}\\
&\bullet \quad D(s\alpha')((s,\alpha)[t,u])=D(s\alpha')u &\label{eq:EITE3}\\
&\bullet\quad D((s,\alpha)[t,u])\leq D(s)&\label{eq:EITE5}
\end{align}
That is, if an algebra is isomorphic to an algebra in $\Pa(X)$ under composition and $D$ and with its Boolean algebra of tests correctly represented as well, and if has a quaternary operation satisfying the above, that quaternary operation must be extended {\em if-then-else}. 
To see this, first notice that these two laws hold for extended {\em if-then-else}.  The first asserts that if one restricts $(s,\alpha)[t,u]$ to values $x\in X$ for which $s(x)\in\alpha$, the result is the same as if $t$ is so restricted.  The second asserts the analogous fact for the case of restricting to $x$ for which $s(x)$ is not in $\alpha$.  The third asserts that the domain of $(s,\alpha)[t,u]$ is no bigger than that of $s$.  All of these laws certainly hold for extended {\em if-then-else}.  Conversely however, only one function satisfies these three conditions: if $v$ is a function satisfying $D(s\alpha)\cdot v=D(s\alpha)t$, $D(s\alpha')\cdot v=D(s\alpha')u$ and $D(v)\leq D(s)$, this guarantees that $v$ is nothing but $(s,\alpha)[t,u]$.  The reason is that the first two equations specify $v$ on $D(s)$ (to agree with either $t$ or $u$), and the third says that it is not defined elsewhere, and precisely one function satisfies these constraints.

So it only remains to correctly axiomatize restriction semigroups of functions having a distinguished Boolean algebra of tests $B$:-- addition of the above laws for extended {\em if-then-else} will then correctly axiomatize the richer structures.

\begin{notation}\label{notation}
We use the symbols $\e,\f,\g,\h$ \up(sometimes with numerical subscripts\up) to denote generic domain elements\up: elements of the form $D(x)$ for some $x$.  Generic elements from the test sort $B$ will typically be a special kind of domain element, and we use lower case Greek letters $\alpha,\beta,\delta$ \up(sometimes with subscripts\up) for these.
\end{notation}

Let $(S,B)$ be such that $S$ is a monoid with zero, having $B$ as a commutative, idempotent submonoid with zero that is equipped with a complementation
operation making it a Boolean algebra (with the semigroup multiplication treated as Boolean meet so that the bottom element is $0$ and the top element is $1$).   Then we say $(S,B)$ is a {\em monoid with tests}.   A {\em submonoid with tests} $(S_1,B_1)$
of the monoid with tests $(S,B)$ has $S_1$ as a submonoid with zero of $S$, and
$B_1$ as a sub-Boolean algebra of $B$; so $(S_1,B_1)$ is itself a monoid with tests.  Note that in this definition there is no requirement that elements of $S_1\backslash B_1$ must lie in $S\backslash B$.  Let $I(X)$ denote the subset of $\mathcal{P}(X)$ consisting of all restrictions of the identity map.  This is a Boolean algebra with respect to the meet operation of composition (which agrees with intersection) and an obvious Boolean complementation: taking $\alpha\in I(X)$ to the identity on the complement of the domain of $\alpha$.  Then $({\mathcal P}(X),I(X))$ is a monoid with
tests and hence so is any submonoid with tests; we call any such {\em functional}.  

\begin{remark}\label{rem:partial}
An \emph{embedding} of a monoid with tests $(S,B)$ into  $({\mathcal P}(X),I(X))$ will be a semigroup embedding $\phi$ of the semigroup $S$ into $\mathcal{P}(X)$ \up(with respect to composition of functions\up) such that the identity of $S$ is mapped by $\phi$ to the identity function of $\mathcal{P}(X)$, and such that $\phi(B)\subseteq I(X)$, with $\phi(\alpha')=\phi(\alpha)'$.  This is equivalent to $(S,B)$ being isomorphic to a submonoid with tests under our definition \up(namely, the monoid $\phi(S)$ with tests from the  Boolean subalgebra $\phi(B)$ of $I(X)$\up).
\end{remark}

Suppose that the monoid with tests $(S,B)$ is such that $S$ is equipped with a unary operation $D$ which satisfies, for all $s,t,u\in S$ and $\alpha,\beta\in B$:
\begin{align}
&\bullet\quad D(s)s=s\quad &\label{eq:D1}\\*
&\bullet \quad D(st)=D(s)D(st) &\label{eq:Dleft}\\
&\bullet\quad D(s)D(t)=D(t)D(s) &\label{eq:Dcom}\\
&\bullet\quad D(D(s))=D(s)\label{eq:DD}\\
&\bullet\quad sD(t)=D(st)s\qquad 
&\label{eq:Dtwisted}\\
&\bullet\quad D(\alpha)=\alpha&\label{eq:DT1}\\
&\bullet\quad
D(s\beta)t=D(s\beta)u \And D(s\beta')t=D(s\beta')u\ \Rightarrow\ D(s)t=D(s)u.&\label{eq:DT2}
\end{align}
Then we call $(S,B)$ a {\em restriction monoid with tests}.  
In terms of programs, these equations are mostly obviously satisfied: for example, the first says that first applying the program which acts like {\em skip} whenever $s$ halts and does not halt otherwise, followed by $s$, gives the same as applying $s$ itself: the two programs halt for the same inputs and give the same answers in such cases.  Law \eqref{eq:Dtwisted} is not obvious, but reflects a general property of partial functions not shared by relations.  The final one results from the fact that for a given (everywhere-defined) test $\beta$, the result of a computation that halts must satisfy either $\beta$ or its complement.

The abstract class of unary semigroups satisfying only (\ref{eq:D1}) to (\ref{eq:Dtwisted}) is the class of restriction semigroups.  Note that if $B=\{0,1\}$, then laws (\ref{eq:DT1}) and (\ref{eq:DT2}) are
trivially satisfied and so contribute nothing; so restriction monoids with tests are generalisations of restriction monoids with zero.

In \cite{csg}, other useful laws are shown to follow: for example $D(st)=D(sD(t))$, as well as $D(s)^2=D(s)$ and $D(D(s)D(t))=D(s)D(t)$,
these two implying that the subset $$D(S)=\{e\in S\mid D(e)=e\}=\{D(s)\mid s\in S\}$$ is a subsemigroup of $S$ which is a semilattice, the elements of which model restrictions of the identity function.  
We view $D(S)$ as a partially ordered set by defining $\e\leq \f$ if $\e=\e\f$, which allows viewing multiplication in $D(S)$ as meet.

We also view $S$ itself as partially ordered, under its {\em natural order} given by $s\leq t \Leftrightarrow s=D(s)t$.
Note that any function semigroup is {\em fundamentally ordered} in the sense of Schein \cite{schein}
(see page 38).  For functions, $f\leq g$ if and only if $f\subseteq g$ when viewed as graphs (sets of ordered pairs).
An important property of the fundamental order is that it is {\em stable}:
\begin{align}
&\bullet\quad s_1\leq t_1, s_2\leq t_2 \Rightarrow s_1s_2\leq t_1t_2.&\label{eq:stable}
\end{align}
The fundamental order is expressible in the language of restriction semigroups, via the natural order, and so (\ref{eq:stable}) will hold automatically.

Note that $({\mathcal P}(X),I(X))$ is a restriction monoid with
tests, and hence so is any submonoid with tests which is closed under $D$; we call any such {\em functional}.  The comments in Remark \ref{rem:partial} apply with obvious modification.

In the proofs to follow, we give many results in the ``test-free" restriction semigroup setting first, if it is possible to do so with no additional effort.  For $S$ a restriction semigroup, let $F\subseteq D(S)$ be a filter (closed under multiplication and such that $\e\in D(S)$ and $\e=\e\f$ implies $\f\in F$). 

\begin{lem}  \label{filtextend}
Let $S$ be a restriction semigroup, with $F$ a proper filter of $D(S)$ and $\h\in D(S)\backslash F$.  Then  
$$F_\h=\{\f\in D(S)\mid \f\geq \g\h, \g\in F\}$$
is a filter of $D(S)$ containing $\h$ and hence properly containing $F$.
\end{lem}
\begin{proof}
Clearly $F_\h$ contains both $\h$ and $F$.  If $f_1\geq g_1h, f_2\geq g_2h$ where $g_1,g_2\in F$, then $f_1f_2\geq gh$ where $g=g_1g_2\in F$, while if $f\geq f_1$ then $f\geq g_1h$ also.  So $F_\h$ is a filter.
\end{proof}

Let $S$ be a restriction semigroup.  Suppose $a,b\in S$ are such that $a\not\leq b$.  We say the filter $F$ of $D(S)$ is {\em $(a,b)$-separating} if (i) $D(a)\in F$, and (ii) there is no $\e\in F$ for which $\e a=\e b$.  

The filter of $D(S)$ generated by $D(a)$ (consisting of all $\e\in D(S)$ for which $D(a)\leq \e$) of course contains $D(a)$.  Suppose it contains $\e\in D(S)$ for which $\e a=\e b$.  Then $a=D(a)a=D(a)\e a=D(a)\e b=D(a)b$, and so $a\leq b$, a contradiction.  So no such $\e\in D(S)$ exists.  So the ordered set of $(a,b)$-separating filters in $D(S)$ is non-empty.

For $a,b\in S$ with $a\not\leq b$, the filter $F$ is {\em maximally $(a,b)$-separating} if it is maximal amongst all $(a,b)$-separating filters in $D(S)$.  Any chain of $(a,b)$-separating filters in $D(S)$ is contained in its union, itself an $(a,b)$-separating filter, and Zorn's Lemma then gives us the following.

\begin{lem} \label{maxsep}
Let $S$ be a a restriction semigroup, with $a,b\in S$ satisfying $a\not\leq b$.  Then there is a maximally $(a,b)$-separating filter in $D(S)$. 
\end{lem}

We use a ``determinative pairs'' approach; the reader is referred to \cite{BStrans} for a detailed introduction to this representation technique, though the current presentation is self contained aside from very routine details.  The determinative pairs technique requires the construction of a right congruence $\epsilon$ of $S$ (that is, an equivalence relation on $S$ that is stable under multiplication on the right) having a block that is a right ideal (that is, is fixed under multiplication on the right).  Our construction of $\epsilon$ will be determined by the restriction semigroup structure.  We use the same definition throughout.

\begin{defn}\label{defn:detpair} Let $S$ be a restriction semigroup, and $F$ a proper filter  of $D(S)$.  Define
$W_F=\{a\in S\mid D(a)\not\in F\}$, a right ideal of $S$.
Define a binary relation $\epsilon_F$ on $S$ by setting
$$a\mathrel{\epsilon_F} b \Leftrightarrow \e a=\e b\mbox{ for some }\e\in F.$$
\end{defn}
It is routine to verify that $\epsilon_F$ is a right congruence on $S$: for all $a,b,c\in S$, $a \mathrel{\epsilon_F} b$ implies $ac\mathrel{\epsilon_F} bc$.  Moreover, $S\backslash W_F$ is a union of $\epsilon_F$-classes since if $x\not\in W_F$ and $(x,y)\in \epsilon_F$, then $\e x=\e y$ for some $\e\in F$, so $D(y)\geq \e D(y)=D(\e y)=D(\e x)=\e D(x)\in F$, so $y\not\in W_F$ also.
Hence $\epsilon_F$ together with $W_F$ constitute a determinative pair in the sense of Schein \cite{BStrans}.  

Let $S_F=(S\backslash \{W_F\})/\epsilon_F$. Now for any $s\in S$, define $\psi^F_s\in {\mathcal P}(S_F)$ by setting
$$\psi^F_s(\bar{x})=\overline{xs}\mbox{ for all }\bar{x}\in S_F\mbox{ for which }xs\not\in W_F,$$ and undefined otherwise, where we are writing $\bar{x}$ for the $\epsilon_F$-class containing $x\in S\backslash W_F$.  The general theory of determinative pairs implies that $\psi^F_s$ is well-defined, and that the mapping $$\theta_F:S\rightarrow {\mathcal P}(S_F),
\mbox{ given by }\theta_F(s)=\psi^F_s\mbox{ for all }s\in S$$ is a semigroup homomorphism mapping any identity element to the identity function and any zero element to the empty function.

Thinking in terms of programs, two programs are equivalent ``modulo the filter $F$" essentially says that the programs act identically on some ``large"' subset of the state space (in a sense determined by the filter $F$) on which both are assumed to act.  A program is in $W_F$ if its domain is not such a large subset: it does not even halt on a large subset.

We now let $S_0$ be the union $\bigcup S_F$ over all maximally $(a,b)$-separating filters $F$ of $D(S)$, ranging over all $a,b$ for which $a\not\leq b$.  Define
$$\theta:S\rightarrow {\mathcal P}(S_0)\mbox{ to be} \bigcup\theta_F,$$ 
the union of the various $\theta_F$ as $F$ ranges over all maximally $(a,b)$-separating filters of $D(S)$. 
It follows that $\theta=\bigcup\theta_F$ is also a semigroup homomorphism since all compositions in ${\mathcal P}(\bigcup S_F)$ are calculated independently on each $S_F$.

The next result implies that every restriction semigroup is functional, which is well-known and which can be proved using a much simpler representation technique.  But in order to represent tests and other operations correctly, the approach used here proves necessary.

\begin{lem}\label{lem:restriction}
The mapping $\theta$ just defined is a restriction semigroup embedding.
\end{lem}
\begin{proof}  Fix a maximally $(a,b)$-separating filter $F$ and define $\theta_F$ as above.
Note that $\psi^F_{D(a)}(\bar{x})=\overline{xD(a)}$, which is defined if and only if
$D(xD(a))=D(xa) \in F$, which is the same as saying that $\psi^F_a(\bar{x})$ is defined.  But then since $D(xa)xD(a)=D(xD(a))xD(a)=xD(a)=D(xa)x$ with $D(xa)\in F$, it follows that $\overline{xD(a)}=\bar{x}$, and so $\psi^F_{D(a)}(\bar{x})=\overline{xD(a)}=\bar{x}$.  So $\psi^F_{D(a)}$ is indeed represented by $\theta_F$ as the restriction of the identity function to the domain of $\psi^F_a$:- $\theta_F(D(a))=D(\theta_F(a))$.  So $\theta_F$ is a restriction semigroup homomorphism.  
It follows that $\theta=\bigcup\theta_F$ is a restriction semigroup homomorphism since, like composition, $D$ on ${\mathcal P}(\bigcup S_F)$ is calculated independently on each $S_F$.  It remains to show that $\theta$ is injective.

Suppose $a\not\leq b$.  If $F$ is a maximally $(a,b)$-separating filter, then using that $F$ to define the $\psi^F_x$, we see that $\psi^F_a(\overline{D(a)})=\overline{D(a)a}=\bar{a}$ is defined since $D(a)\in F$.  If $D(b)\in F$, then also $\psi^F_b(\overline{D(a)})=\overline{D(a)b}$ is defined since $D(D(a)b)=D(a)D(b)\in F$, but $\bar{a}\neq \overline{D(a)b}$ since if it was, then there would be $\e\in F$ for which $\e D(a)a=\e a=\e D(a)b$, and $\e D(a)\in F$, contradicting the fact that there is no $\e\in F$ for which $\e a=\e b$.  If $D(b)\not\in F$, then $\psi^F_b(\overline{D(a)})=\overline{D(a)b}$ is not even defined.  In either case, it must be that $\psi^F_a\not\subseteq \psi^F_b$.  So $\theta(a)\not\subseteq \theta(b)$.
\end{proof}

\begin{thm}  \label{main}
Every restriction monoid with tests is isomorphic to a functional one.
\end{thm}
\begin{proof}
Now assume $S$ is a restriction monoid with tests in the above construction.
It remains to show that $\theta$ correctly represents complementation on $B$ (noting that meet is already represented since $B\subseteq D(S)$).

Now for $\alpha\in B$, first note that $(\psi^F_{\alpha}\circ\psi^F_{\alpha'})(\bar{x})$ is undefined since $x\alpha'\alpha=0$
and $D(0)\not\in F$, so $\psi^F_{\alpha'}$ is the restriction of the identity function to at least some subset of the
complement of the domain of $\psi^F_{\alpha}$.

Conversely, suppose for contradiction that $\bar{x}\in S_F$ is in the domain of neither $\psi^F_{\alpha}$ nor $\psi^F_{\alpha'}$.  So $D(x)\in F$ but $D(x\alpha)\not\in F$ and
$D(x\alpha')\not\in F$.  For any $\h\in D(S)\backslash F$, $F_\h$ as in Lemma~\ref{filtextend} is a filter of $D(S)$ which properly contains $F$.  So there is $\f\in F_{\h}$ such that 
$\f a=\f b$ (since $F_\h$ obviously contains $D(a)$ and $F$ is already maximally $(a,b)$-separating).  So letting $\h=D(x\alpha), D(x\alpha')$ in turn, we see there are $\f_1\geq \g_1 D(x\alpha)$ and $\f_2\geq \g_2 D(x\alpha')$ (for some $\g_1,\g_2\in F$) such that $\f_1 a=\f_1 b$ and $\f_2 a=\f_2 b$.  Then, setting $\g=\g_1\g_2\in F$, we see that $D(\g x\alpha)a=\g D(x\alpha)a=\g D(x\alpha)b=D(\g x\alpha)b$ and similarly that $D(\g x\alpha')a=D(\g x\alpha')b$.  Then law \eqref{eq:DT2} gives $D(\g x)a=D(\g x)b$.  However, $D(\g x)=\g D(x)\in F$ because both $\g$ and $D(x)$ lie in $F$.  This contradicts the choice of $F$ as an $(a,b)$-separating filter.  So the union of the domains of $\psi^F_{\alpha}$ and $\psi^F_{\alpha'}$ is all of $S_F$, as required.
\end{proof}

An {\em extended if-then-else monoid} is a restriction monoid with tests $(S,B)$ having extended
{\em if-then-else} operation $S\times B\times S\times S\rightarrow S$ 
satisfying the laws (\ref{eq:EITE2}), (\ref{eq:EITE3}) and (\ref{eq:EITE5}) above.
$({\mathcal P}(X),I(X))$ is an example, as is any subalgebra of it; we call these {\em functional}.

\begin{cor}
Every extended {\em if-then-else} monoid is up to isomorphism functional.
\end{cor}
\begin{proof}
From Theorem \ref{main}, the restriction monoid with tests reduct of any extended {\em if-then-else} monoid is functional.  But the laws (\ref{eq:EITE2}), (\ref{eq:EITE3}) and (\ref{eq:EITE5}) completely specify extended {\em if-then-else} in functional restriction monoids with tests, as discussed earlier.
\end{proof}

It follows easily that the law $D(s)=(s,1)[1,1]$  holds in all extended {\em if-then-else} monoids (since this holds in all functional examples), and so $D$ can be entirely eliminated from the signature and hence the axioms for extended {\em if-then-else} monoids.
It is also worth noting that the quasi-equational axiom for restriction monoids with tests may be replaced by two equations.

\begin{pro}  \label{simpexit}
In an extended {\em if-then-else} monoid, axiom (\ref{eq:DT2}) is equivalent to the following equational laws:
\begin{align}
&\bullet\quad D(s)t=(s,\alpha)[t,t],\quad &\label{eq:EITE1}\\*
&\bullet\quad (s,\alpha)[t,u]=(s,\alpha)[D(s\alpha)t,D(s\alpha')u]. &\label{eq:EITE4}
\end{align}
\end{pro}
\begin{proof}
If $D(s\beta)t=D(s\beta)u$, then
\bea
D(s)t&=&(s,\beta)[t,t]\\
&=&(s,\beta)[D(s\beta)t,D(s\beta')t]\\
&=&(s,\beta)[D(s\beta)u,D(s\beta')u],
\eea
which by symmetry is $D(s)u$.  The converse follows from the fact that the above laws all hold in functional cases.
\end{proof}

Many further laws now follow easily, for example
$$s\cdot (v,\alpha)[t,u]=(sv,\alpha)[st,su],$$
the $v=1$ case of which is one of the defining laws for the B-semigroups considered in \cite{ITEIJAC}.

Without the introduction of extended \emph{if-then-else}, the law \eqref{eq:DT2} is inherently implicational: it cannot be replaced by one or more equational axioms.  We postpone the proof of this claim until the next section (see Example \ref{eg:quasiv}).

\subsection{Axioms for intersection}

If $s,t\in {\mc P}(X)$, their {\em intersection} $s\cap t$ is in ${\mc P}(X)$ too.  In terms of programs, if $s,t$ are computable, it is clear that so is $s\cap t$: for any input at which both $s,t$ halt, one can create a program that gives the common answer provided by both if they agree and does not halt otherwise.  In terms of weak comparison,  $s*t:=(s=t)[1,0]$ is the restriction of the identity to those $x\in X$ for which $s(x),t(x)$ are both defined and agree, so $s*t=D(s\cap t)$, and $s\cap t=(s*t)s$.  Moreover $D(s)=s*s$.  We first axiomatize $*$ (equivalently intersection) in functional monoids with tests.

Semigroups of functions equipped with both $D$ and $\cap$
were characterized independently in both
\cite{dudtro} and \cite{agree}.  In the latter case, the operation $*$ was used throughout; here are the laws in terms of it.
\begin{align}
&\bullet\quad (s*s)s=s\quad &\label{eq:A1}\\*
&\bullet \quad s*t= t*s &\label{eq:Acom}\\
&\bullet \quad (s*t)s=(s*t)t &\label{eq:Aeq}\\
&\bullet\quad (u*v)s*t=(s*t)(u*v) &\label{eq:Anorm}\\
&\bullet\quad u(s*t)= (us*ut)u&\label{eq:Atwisted}
\end{align}
Any semigroup with a binary operation $*$ satisfying these laws is called {\em twisted agreeable} in \cite{agree}.
It is straightforward to show directly (and in any case it follows from the axioms) that setting $D(s):=s*s$ makes a twisted agreeable semigroup into a restriction semigroup (in which $s*t=D(s*t)$).

The following is Theorem 5.1 of \cite{agree}, noting also that the ``normality'' condition considered there is easily seen to be satisfied in the twisted agreeable case.

\begin{lem}  \label{agreenor}  
In a twisted agreeable semigroup $S$, for all $x,y\in S$, $x*y$ is the largest $\e\in D(S)$ such that $\e\leq D(x)D(y)$ and $\e x=\e y$.
\end{lem}  

This allows us to prove the following lemma (where $\epsilon_F$ and $W_F$ are defined in Definition \ref{defn:detpair}).

\begin{lem} \label{star}
Suppose $S$ is a twisted agreeable semigroup \up(hence an enriched restriction semigroup with $D(s):=s*s$ for all $s\in S$\up).  For any filter $F$ of $D(S)$ and $x,y\in S$, it is the case that $x,y\in S\backslash W_F$ and $(x,y)\in \epsilon_F$ if and only if $x*y\in F$.
\end{lem}
\begin{proof}
If $x,y\in S\backslash W_F$ and $(x,y)\in \epsilon_F$ then $D(x),D(y)\in F$, and $\alpha x=\alpha y$ for some $\alpha\in F$, so $\beta=\alpha D(x)D(y)\in F$ satisfies $\beta x=\beta y$. But $\beta\leq D(x)D(y)$, so $\beta\leq x*y$, and so $x*y\in F$.  

Conversely, if $x*y\in F$, then $D(x),D(y)\in F$, and so $x,y\in S\backslash W_F$, and since $(x*y)x=(x*y)y$, we have $(x,y)\in \epsilon_F$.
\end{proof}

\begin{lem}  \label{*rep}
If $S$ is a twisted agreeable semigroup, the mapping $\theta$ correctly represents $*$.
\end{lem}
\begin{proof}
For each filter $F$, we shall show that $\psi^F_{a*b}=\psi^F_a*\psi^F_b$.  Of course this implies that both sides are restrictions of the identity (since $a*b=D(a*b)$ and $D$ is correctly represented by Lemma \ref{lem:restriction}), so it is only necessary to show that their domains are equal.

Now for $x\in S\backslash W_F$, $\bar{x}\in \dom(\psi^F_{a*b})$ if and only if 
$$(xa*xb)=(D(x)xa*xb)=(xa*xb)D(x)=D((xa*xb)x)=D(x(a*b))\in F,$$ 
which by the previous lemma is equivalent to $xa,xb\in S\backslash W_F$ and $\overline{xa}=\overline{xb}$, or in other words,
$\psi^F_a(\bar{x})=\psi^F_b(\bar{x})$, as required.
\end{proof}

We want to get tests into the picture.
We call a monoid with tests $(S,B)$ which is also a twisted agreeable semigroup in which $B\subseteq D(S)$ a {\em twisted agreeable monoid with tests}.  As usual, $({\mathcal P}(X),I(X))$ is an example, as is any subalgebra of it; we call such examples {\em functional}.

In the presence of $*$, the additional quasi-equation for $D$ may be replaced by an apparently less burdensome one involving only domain elements (recall Notation~\ref{notation}).

\begin{pro}  \label{tam}
In twisted agreeable monoids with tests, law \eqref{eq:DT2} is equivalent to the
following law\up:
\begin{align}
&\bullet\quad D(s\beta)\leq \e,\ D(s\beta')\leq \e \Rightarrow D(s)\leq \e.\quad
&\label{eq:DT2A}
\end{align}
\end{pro}
\begin{proof}
Suppose (\ref{eq:DT2A}) holds but not necessarily (\ref{eq:DT2}).
Suppose $b,c\in S$ are such that $D(a\beta)b=D(a\beta)c$ and $D(a\beta')b=D(a\beta')c$.
Then also, $$D(a\beta)D(b)D(c)b=D(a\beta)D(b)D(c)c$$ and
$$D(a\beta')D(b)D(c)b=D(a\beta')D(b)D(c)c,$$ and so $$D(a\beta)D(b)D(c)\leq b*c,\
D(a\beta')D(b)D(c)\leq b*c,$$ so $$D(D(b)D(c)a\beta)\leq b*c,\
D(D(b)D(c)a\beta')\leq b*c,$$ so $D(D(b)D(c)a)\leq b*c$ (by \eqref{eq:DT2A}), giving $D(a)D(b)D(c)\leq b*c$.

But also, $D(D(a\beta)b)=D(D(a\beta)c)$, so $D(a\beta)D(b)=D(a\beta)D(c)$, and 
similarly $D(a\beta')D(b)=D(a\beta')D(c)$, so $D(a)D(b)=D(a)D(c)$, and so
$D(a)D(b)=D(a)D(c)=D(a)D(b)D(c).$  
Hence
$$D(a)b=D(a)D(b)b=D(a)D(b)D(c)b=D(a)D(b)D(c)(b*c)b,$$
which by symmetry and the fact that $(b*c)b=(b*c)c$ is $D(a)c$ also.  So (\ref{eq:DT2})
holds.

For the converse, if (\ref{eq:DT2}) holds and $D(s\beta)\leq \e$ and $D(s\beta')\leq \e$
then $D(s\beta)\e=D(s\beta)1$ and $D(s\beta')\e=D(s\beta')1$, so $D(s)\e=D(s)1$, that is,
$D(s)\leq \e$, so (\ref{eq:DT2A}) holds.
\end{proof}

From Lemma \ref{*rep} and earlier results we immediately obtain the following.

\begin{cor}  \label{starthm}
Every twisted agreeable monoid with tests is isomorphic to a functional one.
\end{cor}

We now demonstrate that no system of equational axioms can replace law \eqref{eq:DT2} in either the definition of twisted agreeable monoids with tests, or restriction monoids with tests.  The arguments are slightly complicated by the fact that the systems we are considering are two-sorted, with the test sort embedded as a subset of the other elements.  Nevertheless it is clear that equational properties (where some variables may be pre-specified to take values in the test sort) will be preserved under taking a suitable notion of quotient, namely one where the test sort in the quotient consists of classes all of whose elements are congruent to a test element of $A$.  We present a quotient of a twisted agreeable monoid with tests that fails the implication \eqref{eq:DT2}.

\begin{eg}\label{eg:quasiv}
The following system of partial maps on the set $\{0,1,2,\ldots,6,7\}$ forms a twisted agreeable monoid with tests, but has a quotient that fails law~\eqref{eq:DT2}, so is not representable even as a restriction monoid with tests.
\begin{itemize}
\item The empty map $\varnothing$ and the identity map $\id$.
\item The map $s$ with domain $\{0,1,2,3\}$ defined by $x\mapsto x+4$.
\item The test $\beta$ on the set $\{0,1,2,3,4,5\}$ and its complement test $\beta'$ on $\{6,7\}$.  As partial maps, these are restrictions of the identity map to their sets of definition.
\item The function $D(s)$, the identity on restricted domain $\{0,1,2,3\}$.
\item The functions $s\beta$ and $s\beta'$, with domains $\{0,1\}$ and $\{2,3\}$ respectively.
\item A restriction of the identity $\e$, on the set $\{0,1,2\}$.
\item The restrictions of the identity corresponding to $D(s\beta)$, with domain $\{0,1\}$ and $D(s\beta')$, with domain $\{2,3\}$.  We denote $D(s\beta)$ by $\g$ and $D(s\beta')$ by $\f$.
\item The product of $\e\f$ is the restriction of the identity to the set $\{2\}$.
\item The function $\e\f s$, with domain $2$ and $2\mapsto 6$.
\end{itemize}
\end{eg}
\begin{proof}
That the described functions determine a functional twisted agreeable monoid with tests is routine: the system is closed under composition, domain and intersections of functions.  The tests are $\{\varnothing,\beta,
\beta',\id\}$, and the domain elements are these along with $\{\e,\f,\g,\e\f\}$.

Now consider the equivalence relation $\theta$ identifying $\e\f$ with $\f$ and $\e\f s$ with $\f s$, and no other unequal pairs.  We show that $\theta$ is a congruence (no test elements are identified with non-test elements, so the requirement specified above is trivially satisfied).  
First observe that the unary operation $D$ is preserved because $D(\e\f s)=D(\e \f)=\e\f \mathrel{\theta} \f =D(\f)=D(\f s)$.  
Next we verify preservation of intersection.  Consider the nontrivial block $\{\f,\e\f\}$, and observe that relative to the order induced by $\wedge$ (or $D$), we have $\f>\e\f>0$ and that the downset of $\f$ is $\{\f,\e\f,0\}$.  Thus $\theta$ could only fail to be a congruence at this block if there was an $x$ such that $x\wedge \e\f=0$ and $x\wedge \f\in\{\f,\e\f\}$.  But no such $x$ exists, because any element with $x\wedge \e\f=0$ either fixes no points, or has domain omitting both $2$ and $3$.  The argument for the block $\{\f s,\e\f s\}$ is almost identical.

To complete the verification that $\theta$ is a congruence, we now apply a similar argument for composition.  Let $x$ be any element of our model, and observe (by consulting the list of elements) that if $x\f\neq x\e\f$ then $x\f=\f$ and $x\e\f=\e\f$.  Similarly if $x\f s\neq x\e\f s$, then $x\f s=\f s$ and $x\e\f s=\e\f s$.  Thus if the congruence property is to fail, it must fail at a right translation by some $x$.  However if $\f x\neq \e\f x$, then either $\f x=\f$ and $\e\f x=\e\f$ (when $x\in \{\f, D(s),\beta,\id\}$) or $x\in\{s,s\beta'\}$ and we have $\f x=\f s\mathrel{\theta}\e\f s=\e\f x$ (note that $\f s\beta'=\f s$ and $\e\f s\beta'=\e\f s$).  The argument that $\f sx\mathrel{\theta}\e\f sx$ is very similar.  Thus $\theta$ is stable under $D,\wedge,\cdot$.   Moreover, no non-test elements are identified with test elements in $B$, so that the test sort of the quotient is simply the set $B/\theta$ (itself essentially identical to $B$). 

However the quotient by $\theta$, considered either as a restriction monoid with tests or an agreeable monoid with tests, is not itself representable as functions because the law~\eqref{eq:DT2} fails: $D(s\beta)\e=D(s\beta)$ and $D(s\beta')\e=\f\e\mathrel{\theta} \f=D(s\beta')$, but $D(s)\e=\e$, which is not congruent to $D(s)$.
\end{proof}

\subsection{Axioms for weak comparison}

Note that in ${\mathcal P}(X)$, $(f\neq g):=(f=g)[0,1]$ restricts the identity to where $f,g$ disagree (but are both defined).  On the other hand, for functions $f,g,h,k$ we may write
$$(f=g)[h,k]=(f*g)h\cup (f\neq g)k.$$
We next axiomatize $*,\neq$ in semigroups of functions and hence functional monoids with tests, which easily leads to axioms for weak comparison itself.  

Call a twisted agreeable semigroup on which there is a binary operation $\neq$ satisfying the laws below a {\em disagreeable semigroup}.  

\begin{align}
&\bullet\quad D(s\neq t)=(s\neq t)\quad &\label{eq:in1}\\*
&\bullet \quad s(t\neq u)=(st\neq su)s &\label{eq:intwist}\\
&\bullet \quad (s*t)(s\neq t)=0 &\label{eq:ineq}\\
&\bullet \quad \e(u\neq v)=(\e u\neq \e v) &\label{eq:innorm}\\
&\bullet\quad (s*t)\leq \e,(s\neq t)\leq \e\ \Rightarrow D(s)D(t)\leq \e &\label{eq:inimp}
\end{align}

It is routine to verify that $({\mathcal P}(X),I(X))$ is an example (only law \eqref{eq:intwist} requiring any real checking), as is any subalgebra, and we call these subalgebras {\em functional} (see Remark \ref{rem:partial}).

\begin{lem}  \label{maxagree}
If $S$ is a twisted agreeable semigroup with $a,b\in S$ satisfying $a\not\leq b$, and $F$ is a filter, the following are equivalent.
\ben
\item $F$ is maximally $(a,b)$-separating.
\item $F$ is maximal with respect to: $D(a)\in F$ but $a*b\not\in F$.
\een
\end{lem}
\begin{proof}
$(\Rightarrow)$.  If $F$ is maximally $(a,b)$-separating, then obviously $D(a)\in F$ and $a*b\not\in F$.  Let the filter $G$ properly contain $F$.  Then there is $\e\in G$ such that $\e a=\e b$, and $D(a)\in G$.  So 
\[
D(b)\geq \e D(b) =D(\e b)=D(\e a)=\e D(a)\in G,
\]
and so $D(b)\in G$; hence $a*b\geq \e D(a)D(b) \in G$ by Lemma \ref{agreenor}.  Hence $F$ is also maximal with respect to $D(a)\in F$ and $a*b\not\in F$.  

$(\Leftarrow)$.  Suppose $F$ is maximal with respect to $D(a)\in F$ but $a*b\not\in F$.  So any filter properly containing it must contain $D(a)$ and $a*b$, hence is not $(a,b)$-separating, so $F$ is maximally $(a,b)$-separating.   
\end{proof}

A twisted agreeable monoid with tests which is a disagreeable semigroup is a {\em disagreeable monoid with tests}.

\begin{thm}  \label{disrep}
Every disagreeable semigroup is functional.
\end{thm}
\begin{proof}
Again, we want to show that each $\theta_F$ represents $\neq$ correctly, which means showing that for a given
$F$ which maximally separates $a,b\in S$ for which $a\not\leq b$, we have that $(\psi^F_c\neq \psi^F_d)=\psi^F_{c\neq d}$ for all $c,d\in S$.

Again, both are restrictions of the identity (by law \eqref{eq:in1}), so we must show their domains coincide.  Now for $\bar{x}\in \dom(\psi^F_{c\neq d})$, we have that $\psi^F_{c\neq d}(\bar{x})=\overline{x(c\neq d)}$, where $(xc\neq xd)=D((xc\neq xd)x)=D(x(c\neq d))\in F$.  On the other hand, $\psi^F_c(\bar{x})\neq \psi^F_d(\bar{x})$ says that
$\overline{xc}\neq\overline{xd}$ (but that $D(xc),D(xd)\in F$).  So by Lemma \ref{star}, we want to show that
$(xc*xd)\not\in F$ if and only if $(xc\neq xd)\in F$.

Now $(xc*xd)(xc\neq xd)=0\not\in F$, so not both can be in $F$.  So $(xc\neq xd)\in F$ implies $(xc*xd)\not\in F$.
Conversely, suppose that $(xc*xd)\not\in F$, and to obtain a contradiction, that $(xc\neq xd)\not\in F$ also.  So $F_{xc*xd}$ as in Lemma \ref{filtextend} properly contains $F$, and so by Lemma \ref{maxagree}, $a*b\in F_{xc*xd}$, and so $a*b\geq (xc*xd)\g_1$ for some $\g_1\in F$.  Similarly, since  $(xc\neq xd)\not\in F$, we can conclude that $a*b\geq (xc\neq xd)\g_2$ for some $\g_2\in F$.

Letting $\g=\g_1\g_2\in F$, we have $a*b\geq \g(xc*xd)=(\g xc*\g xd)$ by (\ref{eq:Anorm}), and similarly $a*b\geq (\g xc\neq \g xd)$ by (\ref{eq:innorm}), so by (\ref{eq:inimp}), $a*b\geq D(\g xc)D(\g xd)=\g D(xc)D(xd)\in F$ since each of the three terms in the product is in $F$, and so $a*b\in F$, so $F$ fails to separate $a,b$, a contradiction.  So $(xc\neq xd)\in F$ as required.
\end{proof}

\begin{cor}  \label{neqthm}
Every disagreeable monoid with tests is functional.
\end{cor}

The following example demonstrates that without weak comparison, the implication \eqref{eq:inimp} cannot be replaced by equational laws in the definition of a disagreeable monoid with tests.
\begin{eg}
Consider the following system of partial maps on the set $\{0,\ldots,9\}$\up.
\begin{itemize}
\item $s$, with domain $\{0,1,2,3,4\}$ and $x\mapsto x+5$.
\item $t$, with domain $\{0,1,2,3,4\}$ and with $0\mapsto 5$, $1\mapsto 7$, $2\mapsto 6$, $3\mapsto 9$, $4\mapsto 8$.
\item $s\wedge t$ with domain $\{0\}$ and $0\mapsto 5$.
\item $D(s)=D(t)$, the identity on $\{0,1,2,3,4\}$.
\item $\f$, the identity on $\{0\}$.  Note that $\f=s*t=D(s\wedge t)$.
\item $\g$, the identity on $\{1,2,3,4\}$.  Note that $\g=(s\neq t)$.
\item $\e$, the restriction of the identity to $\{0,1,2\}$.
\item $\e\g$, the identity on $\{1,2\}$.
\item $\e s$, $\e t$, $\f s$, $\g s$, $\g t$, $\e \g s$, $\e\g t$.  Note that $\f s=\f t$ as $\f=s*t$.
\item The two tests $\varnothing$ and $\id$.
\end{itemize}
These form a disagreeable monoid with tests, but there is a quotient failing law \eqref{eq:inimp}.
\end{eg}
\begin{proof}
Let $\theta$ be the equivalence identifying the three nontrivial pairs  $(\g,\e\g)$, $(\g s,\e\g s)$ and $(\g t,\e\g t)$.  We first show that $\theta$ is a congruence.  Throughout the following, let $(x,y)$ be one of the three (up to symmetry) nontrivial pairs in $\theta$.

For stability under $\cdot$, 
begin by observing that the domain of any element in one of the three non-trivial pairs is either $\{1,2,3,4\}$ (for $\g,\g s,\g t$) or $\{1,2\}$ (for $\e\g,\e\g s,\e\g t$).  If the range of an element $z$ does not overlap with at least one of these sets, then $zx=0=zy$.  But the only elements $z$ whose range nontrivially intersects either of these sets are $\g$ and $\e\g$ (and the trivial case of $\id$), which are restrictions of the identity and lead to $zx\mathrel{\theta} x\mathrel{\theta} y\mathrel{\theta}zy$ in all cases.  For right multiplication by an element $z$, we obtain $(xz,yz)\in\{(\g,\e\g),(\g s,\e\g s),(\g t,\e\g t)\}$ when $(x,y)=(\g,\e\g)$, and  $(x,y)\in\{(\g s,\e\g s),(\g t,\e\g t)\}$ when $(x,y)\in\{(\g s,\e\g s),(\g t,\e \g t)\}$.  These are subsets of $\theta$ so that preservation of $\theta$ by $\cdot$ is verified.

Next we consider $*$ and $\neq$.  As the domains of $x$ and $y$ are either $\{1,2\}$ or $\{1,2,3,4\}$ it follows that for any element $z$ we have $z*x,z*y\in\{0,\g,\e\g\}$ and $z\neq x,z\neq y\in\{0,\g,\e\g\}$.  However a quick examination of the elements of our example reveals that if $z$ agrees (or disagrees) with one of $x$ or $y$ on $\{1,2\}$ then it agrees (or disagrees) with both $x$ and $y$ on $\{1,2\}$.  So in fact either $\{z*x,z*y\}=\{0\}$ or $\{z*x,z*y\}=\{\g,\e\g\}$, and similarly for $\{z\neq x,z\neq y\}$.  As these are blocks of $\theta$, it follows that $\theta$ is preserved by $*$ and $\neq$.  Thus $\theta$ is a congruence; there is no identification of non-test elements with test elements, so that in the quotient, the test sort unambiguously consists of the two singleton classes $\{\varnothing\}$ and $\{\id\}$.

Now we demonstrate that law \eqref{eq:inimp} fails in the quotient by $\theta$.   Now  $(s*t)\leq \e$ and $(s\neq t)=\g\mathrel{\theta}\e\g\leq \e$, and law \eqref{eq:inimp} would require $D(s)D(t)\leq \e$ modulo $\theta$.  However this fails, so that the quotient is not representable as functions.
\end{proof}

A {\em weak comparison monoid} is a monoid $S$ with zero $0$  equipped with a weak comparison
operation $S\times S\times S\times S\rightarrow S$ which is such that $s*t:=(s=t)[1,0]$ and $s\neq t:=(s=t)[0,1]$ define a disagreeable monoid on $(S,B)$, also
satisfying the following three laws.
\begin{align}
&\bullet \quad (s*t)\cdot (s=t)[u,v]=(s*t)u &\label{eq:comp2}\\*
&\bullet \quad (s\neq t)\cdot (s=t)[u,v]=(s\neq t)v &\label{eq:comp3}\\
&\bullet\quad D((s=t)[u,v])\leq D(s)D(t).\quad &\label{eq:comp1}
\end{align}
A {\em weak comparison monoid with tests} is a monoid with tests $(S,B)$ such that $S$ is a weak comparison monoid.
$({\mathcal P}(X),I(X))$ is an example of a weak comparison monoid and indeed a weak comparison monoid with tests, as is any subalgebra of it; as in previous cases, we call these {\em functional}.

\begin{cor}  \label{corweak}
Every weak comparison monoid \up(possibly with tests\up) is up to isomorphism functional.
\end{cor}
\begin{proof}
Lemma \ref{*rep} and Theorem \ref{disrep} imply that the disagreeable monoid reduct of any weak comparison monoid
is functional; now note that laws (\ref{eq:comp2}), (\ref{eq:comp3}) and (\ref{eq:comp1}) completely specify weak comparison in functional disagreeable semigroups.  The extension to the case with tests is immediate.
\end{proof}

Analogous to Proposition \ref{simpexit}, the quasi-equational axiom for $\neq$, namely (\ref{eq:inimp}), can be expressed
equationally in the presence of weak comparison.

\begin{pro} \label{simpcomp}
Within weak comparison monoids, quasi-equational law (\ref{eq:inimp}) for disagreeable monoids is equivalent to the following equational laws:
\begin{align}
&\bullet \quad (s=t)[u,u]=D(s)D(t)u\quad &\label{eq:wc1}\\*
&\bullet \quad (s=t)[u,v]=(s=t)[(s*t)u,(s\neq t)v] &\label{eq:wc2}
\end{align}
\end{pro}
\begin{proof} 
If the above two laws hold and $(s*t)\leq \e,(s\neq t)\leq \e$, then
\bea
D(s)D(t)\e&=&(s=t)[\e,\e]\\
&=&(s=t)[(s*t)\e,(s\neq t)\e]\\
&=&(s=t)[(s*t),(s\neq t)]\\
&=&(s=t)[1,1]\\
&=&D(s)D(t),
\eea
so $D(s)D(t)\leq \e$.
Conversely, the above two laws clearly hold in functional cases.
\end{proof}

${\mathcal T}(X)$ is closed under weak comparison, where the operation is called {\em comparison} in \cite{compsemi}, following Kennison \cite{kennison}.
Algebras of transformations under composition and comparison
are axiomatized in \cite{compsemi} in a test-free setting.  An axiomatization for partial functions is also given \cite{compsemi}, but with a different ``non-computable" interpretation of the comparison operation in which agreement includes places
where both functions are undefined, an interpretation suited to obtaining a rich algebra of partial transformations but not relevant for current purposes where the goal is to model possibly non-halting computable functions.  For this reason we have used the phrase ``weak comparison" here, to distinguish it from this previously used definition for partial functions.

\section{Some special cases}

\subsection{$B=D(S)$ and modal restriction semigroups}

The largest $B$ can be is all of $D(S)$.
In that case, the operations considered here reduce
to those axiomatized in \cite{modrest}.  Specifically, the operations $P,\bowtie$ and $\sqcup$ discussed in \cite{modrest} are defined on ${\mathcal P}(X)$ as follows:
\bi
\item $P(s)$ is the restriction of the identity function to the complement of the domain of $s$; that is, $P(s):=D(s)'$;
\item $(s\bowtie t):=(s*t)\cup D(s)'D(t)'$, the restriction of the identity function to those $x\in X$ where $s,t$ do not disagree;
\item $s\sqcup t:=D(s)[s,t]$, which is $s$ where it is defined together with $t$ when $s$ is not defined (and undefined otherwise).
\ei
Conversely we may write:
\bi
\item $s*t:=(s\bowtie t)D(s)D(t)$;
\item $s\neq t:=(s\bowtie t)'$;
\item $(s,\alpha)[t,u]:=D(s\alpha)t\sqcup D(s\alpha')u$;
\item $(s=t)[u,v]:=(s*t)u\sqcup (s\neq t)v$;
\item $\e\cup \f:= \e\vee \f$ for $\e,\f\in D(S)=B$.
\ei

\subsection{The test-free algebra of non-halting programs}

At the other end of the spectrum, the smallest $B$ can be in a restriction monoid with tests is $\{0,1\}$.  This case has some interest in terms of modelling the {\em test-free} algebra of computable functions.   The extended {\em if-then-else} operations are of no interest, but the disagreeable operation and weak comparison still make sense.  Indeed, many of our earlier results were proved at this level of generality, specifically Theorem \ref{disrep} and Corollary \ref{corweak}; neither the disagreeable operation nor weak comparison had previously been axiomatized in any function semigroup setting, as far as we know.

\subsection{Relation to B-semigroups}

In \cite{ITEIJAC}, the class of B-semigroups $(S,B)$ is shown to (finitely) axiomatize the class of transformation semigroups equipped with {\em if-then-else} operations indexed by a Boolean algebra.  These arise as the subalgebras of reducts of extended {\em if-then-else}-monoids consisting of all elements $s$ satisfying $D(s)=1$, in which the test elements are assumed to be part of a distinct sort and closed under the mapping $\alpha\mapsto D(s\alpha)$ for all $s\in S$.

\section{Extensions, enrichments and some open problems}

\subsection{Tentative axioms for extended {\em while-do}}  \label{while}

It would be remiss to say nothing further about looping here, since this is the obvious source of the non-halting of programs currently
being modelled!  Axiomatizing the {\em while-do} command is a very difficult problem, even when only halting tests are considered, and it is unlikely that a finite
axiomatization exists, at least if completeness with respect to functionally valid implications is desired.  However, with relatively little effort, we can
obtain a reasonable first approximation.  In \cite{modrest}, $(\alpha:s)$ is defined and crudely axiomatized for the case in which $B=D(S)$, and we adopt a similar approach here.

We say the extended {\em if-then-else} monoid $(S,B)$ is a {\em W-monoid} if it is equipped with a set of mixed ternary operations
$S\times B\times S\rightarrow S$ obeying the following law:
\begin{align}
&\bullet\quad ((t,\alpha):s)=(t,\alpha)[s((t,\alpha):s),1].\quad &\label{eq:W12}
\end{align}
So $({\mathcal P}(X),I(X))$ is a W-monoid if we define $((f,\alpha):g)$ to be extended {\em while-do} as discussed earlier. Another way to state (\ref{eq:W12}) is as the two separate laws:
\begin{align}
&\bullet\quad D(t\alpha)((t,\alpha):s)=D(t\alpha)s((t,\alpha):s),\quad &\label{eq:W1}\\*
&\bullet\quad D(t\alpha')((t,\alpha):s)=D(t\alpha').\quad &\label{eq:W2}
\end{align}

\begin{lem}  \label{powers}
In the restriction semigroup $S$, if $\e,\f\in D(S)$ and $s,t\in S$ satisfy
$\e st=\e t$ and $\f t=\f$, then $(\e s)^n\f\leq t$ for all $n\geq 0$.
\end{lem}
\begin{proof} We use induction on $n$.  Now $\f t=\f$ implies that $\f\leq t$, giving the $n=0$ case.  Assuming the $n=k$ case, we have
$$(\e s)^{k+1}\f=(\e s)(\e s)^k\f\leq \e st=\e t\leq t $$
by the stability property, and the result follows.
\end{proof}

Lemma \ref{powers} shows that any element $t$ in a W-monoid satisfying both $D(t\alpha) su=D(t\alpha)u$ and
$D(t\alpha')u=D(t\alpha')$ will necessarily be at least as big as each element
$(D(t\alpha) s)^nD(t\alpha')$ for $n$ a natural number; in particular, this is
true of $((t,\alpha):s)$.  With the help of some additional laws, we can do better.

We say the W-monoid $S$ is {\em Kleenean} if for all $s,t,u,\alpha$:
\bi
\item $((t,\alpha):s)D(t\alpha')=((t,\alpha):s)$ and
\item $D(t\alpha)su\leq u\ \Rightarrow \ ((t,\alpha):s)u\leq u$.
\ei
The second rule above is analogous to a rule for Kleene algebras (possibly without tests), namely $su\leq u\ \Rightarrow\ s^*u\leq u$.

The structure $({\mathcal P}(X),I(X))$ is Kleenean, and $((t,\alpha):s)$ is the
smallest element at least as big as $(D(t\alpha) s)^nD(t\alpha')$ for all $n\geq 0$ (indeed it is their disjoint union), hence must also be the smallest element $u$ amongst those satisfying $D(t\alpha) su=D(t\alpha) u$ and $D(t\alpha')u=D(t\alpha')$.  In general we have the following.

\begin{pro}  \label{minb}
Let $S$ be a functional Kleenean restriction W-monoid with tests.  Then for all $s,t,\alpha$,
$((t,\alpha):s)$ is the smallest $u$ for which $D(t\alpha) su=D(t\alpha) s$ and $D(t\alpha')u=D(t\alpha')$.
\end{pro}
\begin{proof} If $u$ is one such, then $D(t\alpha) su\leq u$ and so $((t,\alpha):s)u\leq u$, so that
$$((t,\alpha):s)=((t,\alpha):s)D(t\alpha')=((t,\alpha):s)D(t\alpha')u=((t,\alpha):s)u\leq u.$$
Conversely, $((t,\alpha):s)$ is one such $u$, as we have seen.
\end{proof}

Proposition \ref{minb} shows that in any representation $\psi$ of a functional Kleenean restriction W-monoid with tests, the representation of $((t,\alpha):s)$ is correct relative to the image of $\psi$.
Something analogous happens with the Kleene closure of an element in a Kleene algebra (with or without tests): $r^*$ is the least ``reflexive transitive element" in the algebra containing $r$, and any representation in terms of relations
will represent it as the least reflexive transitive relation containing $r$ amongst those relations in the algebra.  

Recall that a semigroup is \emph{periodic} if for every element $x$, there are positive integers $i$ and $p$ such that $x^i=x^{i+p}$.  The following result is essentially a corollary of Lemma \ref{powers} and Proposition \ref{minb}.  
\begin{thm}
Let $S$ be a functional restriction monoid with tests also carrying extended if-then-else operation satisfying laws \eqref{eq:EITE2}--\eqref{eq:EITE5} and extended while-do operations satisfying the Kleenean W-monoid axioms.  If $S$ is periodic, then any functional representation as a monoid with tests, correctly represents both extended if-then-else and extended while-do.
\end{thm}
\begin{proof}
Assume $S$ has been represented over some set $X$ by a restriction monoid with tests representation $\theta$.  Correct representability of extended if-then-else is observed in Corollary \ref{ite} (it follows because of the observation that properties \eqref{eq:EITE2}--\eqref{eq:EITE5} define extended if-then-else in terms of composition, tests and domain).  Correct representability of extended while-do will follow because in the periodic case, an extended while-do can be written as a finite number of nested extended if-then-else statements, and such an element is correctly represented.

To make this intuitive idea rigorous, observe that for any $t,\alpha,s$ the correct functional representation of $((t^\theta,\alpha^\theta):s^\theta)$ is $\bigcup_{i\in\omega}\big((D(t\alpha) s)^iD(t\alpha')\big)^\theta$ but with the assumption of periodicity, this infinite union coincides with the finite union $\bigcup_{i\leq n}\big((D(t\alpha) s)^iD(t\alpha')\big)^\theta$ for some $n$.  We will give an explicit description of a nested series of extended if-then-else statements $v$ in $S$, which will be represented as the function $\bigcup_{i\leq n}\big((D(t\alpha) s)^iD(t\alpha')\big)^\theta$.  This will be sufficient to show that $v=((t,\alpha):s)$ (and therefore that $((t,\alpha):s)^\theta=((t^\theta,\alpha^\theta):s^\theta)$ as required), because $v$ satisfies $(D(t\alpha) s)^\theta v^\theta=(D(t\alpha) s)^\theta$ and $(D(t\alpha'))^\theta v^\theta=(D(t\alpha'))^\theta$ so has $v\geq ((t,\alpha):s)$ by Proposition~\ref{minb} (and the fact that $\theta$ is a faithful restriction monoid with tests representation).  But also $v\leq ((t,\alpha):s)$ because $\big((D(t\alpha) s)^iD(t\alpha')\big)^\theta\subseteq ((t,\alpha):s)^\theta$ for every $i$ by Lemma~\ref{powers}.

We now inductively define the nested if-then-else statement $v$.  Let $v_0$ denote the element $(D(t\alpha)s)^{n}D(t\alpha')$.  
Now assume that we have defined $v_k$ for some $0\leq k\leq n-1$ and that
\begin{equation}
v_k^\theta=\bigcup_{n-k\leq i\leq n}((D(t\alpha)s)^{i}D(t\alpha'))^\theta.\label{eq:unionversion}
\end{equation}
Define 
\[
v_{k+1}:=((D(t\alpha)s)^{n-(k+1)}t,\alpha)[v_k,(D(t\alpha)s)^{n-(k+1)}D(t\alpha')]
\]
The definition of extended if-then-else shows (for $k>0$) that
$v_{k+1}^\theta$ is equal to the union of the representation of $D\big((D(t\alpha)s)^{n-(k+1)}t\alpha\big)v_{k}$ with the representation of 
\begin{equation}
D\big((D(t\alpha)s)^{n-(k+1)}t\alpha'\big)\big(D(t\alpha)s\big)^{n-(k+1)}D(t\alpha').\label{eg:alongtheway}
\end{equation}
As $D\big((D(t\alpha)s)^{n-(k+1)}t\alpha'\big)=D\big((D(t\alpha)s)^{n-(k+1)}D(t\alpha')\big)$, the expression \eqref{eg:alongtheway} simply reduces to $(D(t\alpha)s)^{n-(k+1)}D(t\alpha')$.
Similarly, 
\[
D((D(t\alpha)s)^{n-(k+1)}t\alpha)=D((D(t\alpha)s)^{n-(k+1)}D(t\alpha)),
\]
so that the induction hypothesis \eqref{eq:unionversion} gives $D((D(t\alpha)s)^{n-(k+1)}t\alpha)v_{k}=v_k$.  Thus $v_{k+1}^\theta$ is the union of $v_k^\theta$ with $((D(t\alpha)s)^{n-(k+1)}D(t\alpha'))^\theta$, showing that the induction hypothesis is preserved.  In particular this shows that $v_n$ is the desired union $\bigcup_{i\leq n}\big((D(t\alpha) s)^iD(t\alpha')\big)^\theta$.
 \end{proof}

If $S$ is finite, then it is periodic, and moreover the representation method used above represents $S$ as functions on a finite set.  So a corollary to the above is that for finite $S$, the Kleenean restriction W-monoid with if-then-else axioms are sound and complete for functional models on finite sets.

There is also the possibility of defining extended {\em while-do} operations in terms of the equality (partial) predicate: one could define one or both of $((f=g):h)$ and $((f\neq g):h)$, and easy analogs of the above definitions and results may be obtained.  
\begin{problem}
Is there a finite axiomatisation that is complete for equational properties of restriction semigroups of functions equipped with extended {\em while-do} and extended \emph{if-then-else}?  Is there a finite complete axiomatization for the quasiequational theory?  Is there even a recursively enumerable and complete axiomatization\footnote{Some authors prefer to include ``recursively enumerable'' as part of the definition of ``axiomatisation'', in which case we are asking whether or not there is a complete axiomatisation.} for the quasiequational theory?
\end{problem}

\subsection{Non-halting tests}

Non-halting tests were considered by Manes in \cite{manes}, where {\em if-then-else} algebras over Boolean algebras,
C-algebras and ADAs were considered, in the absence of composition.  Here, C-algebras and ADAs are algebras of non-halting conditions, generalising Boolean algebras.  In the current setting, our tests are assumed to form a Boolean algebra, although it turns out that we can construct an algebra of ``non-halting tests" from these Boolean tests together with some of our operations.

First note that the structure of the Boolean algebra $B$ is faithfully captured by its induced {\em if-then-else}  action:
$\alpha[x,y]=\beta[x,y]$ for all $x,y\in S$ if and only if $\alpha=\beta$ (as follows on setting $x=1,y=0$ and then $x=0,y=1$).  However, this reflects the fact that we choose to distinguish elements of $B$ based only on their effect in {\em if-then-else} statements.

Similarly, letting $P[x,y],Q[x,y]$ be induced binary operations of the form $(a=b)[x,y]$ or $(a,\alpha)[x,y]$, or indeed $\alpha[x,y]$
(letting $a=1$ in the previous case), we can recursively generate new operators by setting
\begin{align}
&\bullet\quad (P\wedge Q)[x,y]:=P[Q[x,y],y]\quad &\label{NH1}\\*
&\bullet \quad (P\vee Q)[x,y]:=P[x,Q[x,y]] &\label{NH2}\\
&\bullet\quad (\lnot P)[x,y]:=P[y,x]&\label{NH3}
\end{align}
for all $x,y\in S$.  Identifying each such $P$ with its functional effect by setting $P=Q$ if and only if $P[x,y]=Q[x,y]$
for all $x,y\in S$, we see that the above recursive scheme generates an algebra of ``generalised predicates" $B^*$ under
$\wedge,\vee,\lnot$, in which $B$ is embedded as a subalgebra.

It is easily checked that these induced logical operations on $B^*$ have the following interpretations in functional cases:
\bi
\item $P\wedge Q$ is {\em true} if both $P,Q$ are; {\em false} if $P$ is {\em false}, or if $P$ is defined and $Q$ is
{\em false}; and undefined otherwise.
\item $P\vee Q$ is {\em true} if $P$ is {\em true}, or if $P$ is defined and $Q$ is {\em true}; {\em false} if both $P,Q$
are {\em false}; and undefined otherwise.
\item $\lnot P$ is {\em true} if and only if $P$ is {\em false}, and vice versa, and is undefined if $P$ is.
\ei
In \cite{manes}, Manes considers exactly these connectives on ``non-halting" conditions, and justifies them in terms of the way actual programming languages work.  It follows that $(B^*,\wedge,\vee,\lnot,0,1)$ is a {\em C-algebra} in the sense of \cite{manes}, where ITE-algebras over C-algebras are studied. 

In \cite{manes}, the C-algebras of non-halting tests were assumed to have some prior independent existence (satisfying various laws generalising Boolean algebra), and indeed
the above three connective definitions (\ref{NH1}) to (\ref{NH3}) were assumed to hold as laws for the ITE-algebras considered in \cite{manes}.
In the current setting too, an alternative approach would be to begin with an abstract collection of non-halting conditions as in \cite{manes}, rather than deriving them from a given Boolean algebra of ``elementary" halting conditions via extended {\em if-then-else}.  
\begin{problem}
Characterise the algebras of computable functions associated with an abstract C-algebra of non-halting tests.
\end{problem}

\subsection{Complexity and decidability}
Hardin and Kozen \cite{harkoz} showed that the implicational theory of relational models of KAT is $\Pi_1^1$-complete: so no complete recursive axiomatisation is possible.  However, the equational theory is well known to be only PSPACE.  Goldblatt and the first author \cite{goljac} showed that deciding functional validity for propositions in strict fragments of deterministic PDL is $\Pi_1^1$-hard, which can be translated to the algebraic approach taken here in the form of the $\Pi_1^1$-hardness of the equational theory of functionally representable algebras in the language containing composition, antidomain, intersection and {\em while}.  The argument in \cite{goljac} makes intrinsic use of the ability to nest halting statements as test conditions for other halting statements (which is enabled by antidomain, or equivalently, the modal necessity operator of PDL). Such nesting is impossible in the signatures considered in the present article.
\begin{problem}
What is the complexity of the equational theory of representable algebras in the various signatures considered here, when {\em while} is included.  In particular, is it possible that the equational theory is decidable for the class of functionally representable  Kleenean restriction W-monoids with tests?    Under what constructions does the implicational theory of while achieve high undecidability \up(such as $\Pi_1^1$-hardness\up{)?}
\end{problem}

\end{document}